\documentclass[a4paper,reqno,11pt]{amsart}
\usepackage{amsmath}
\usepackage{amssymb}
\usepackage{cases}
\usepackage{enumitem}
\usepackage{hyperref}
\hypersetup{
	colorlinks=true,
	citecolor=black,
	linkcolor=black,
	urlcolor=green,
}
\allowdisplaybreaks[4]
\newtheorem{thm}{Theorem}[section]
\newtheorem{prop}[thm]{Proposition}

\newtheorem{lem}[thm]{Lemma}

\newtheorem{cor}[thm]{Corollary}

\newtheorem{conj}[thm]{Conjecture}
\newtheorem{claim}[thm]{Claim}

\theoremstyle{definition}
\newtheorem{definition}[thm]{Definition}
\newtheorem{example}[thm]{Example}

\theoremstyle{remark}
\newtheorem{remark}[thm]{Remark}

\numberwithin{equation}{section}

\newcommand{\bR}{\mathbb{R}}

\newcommand{\bP}{\mathbb{P}}

\newcommand\OO{{\mathcal{O}}}
\newcommand\0{{\mathbf{0}}}
\newcommand\n{{\mathbf{N}}}

\newcommand\ZZ{{\mathbb{Z}}}

\newcommand{\emb}{\operatorname{emb}}
\newcommand{\mld}{\operatorname{mld}}
\newcommand{\inte}{\operatorname{int}}

\newcommand{\Vol}{\operatorname{Vol}}
\newcommand{\vol}{\operatorname{vol}}
\newcommand{\vex}{\operatorname{vert}}

\newcommand{\conv}{\operatorname{conv}}
\newcommand{\ps}{\operatorname{PS[q]}}

\usepackage{todonotes}

\begin{document}

\title{Optimal upper bounds for anti-canonical volumes of singular toric Fano varieties}
\date{\today}

\author{Yu Zou}
\address{Yu Zou, College of Mathematics and Statistics, Chongqing University, Chongqing, 401331, China}
\email{fishlinazy@cqu.edu.cn}

\begin{abstract}
Fix two positive integers $d\geq3$ and $q$. We give an upper bound for anti-canonical volumes of $d$-dimensional $\frac{1}{q}$-lc toric Fano varieties, which corresponds to an upper bound for the dual normalized volumes of the associated $d$-dimensional $\frac{1}{q}$-lc Fano polytopes. And we also construct examples to show that these upper bounds are optimal. Besides, we provide an optimal upper bound for volumes of $d$-dimensional lattice simplices $S$ such that $\frac{1}{q}S$ has exactly one interior lattice point.
\end{abstract}


\maketitle
\pagestyle{myheadings} \markboth{\hfill Y.~Zou
\hfill}{\hfill Optimal upper bounds for anti-canonical volumes \hfill}

\tableofcontents

\section{Introduction}
\subsection{Upper bounds for anti-canonical volumes of Fano varieties}
Fano varieties are among one of the most important study objects in algebraic geometry. Fix a positive integer $d$ and a positive real number $\epsilon$. It is known that $d$-dimensional $\epsilon$-lc Fano varieties form a bounded family (see \cite{Bir21} and see \cite{BB93} for toric Fano varieties). Motivated by the classification of Fano varieties, it is very important to find the optimal upper bound for anti-canonical volumes of $d$-dimensional $\epsilon$-lc Fano varieties. 

For dimensions $d\leq2$, the optimal upper bounds have already been known (see \cite{Jiang13} for the surface case). For dimension $d=3$, \cite{JZ24} gives the upper bound $\frac{3200}{\epsilon^4}$ for the anti-canonical volume of any $\epsilon$-lc Fano threefold when $0<\epsilon<\frac{1}{3}$. Though this upper bound is not optimal (which can be seen from the proof of this result), the order $\OO(\frac{1}{\epsilon^4})$ is sharp due to \cite[Example~6.3]{Ambro16}. So it is reasonably to conjecture that, in each dimension $d\geq4$, the optimal upper bound for the anti-canonical volumes of $d$-dimensional $\epsilon$-lc Fano varieties has the same order with $\frac{1}{\epsilon^{2^d-d-1}}$. 

For two positive integers $d\geq3$ and $q$, the main goal of this paper is to figure out an optimal upper bound for the anti-canonical volumes of $d$-dimensional $\frac{1}{q}$-lc toric Fano varieties, with order $\OO(q^{2^{d}-d-1})$. 

Let $u_{i,q} (i\in\ZZ_{>0})$ be the integer number defined recursively as the following:$$u_{1, q}=q, u_{k+1, q}=u_{k, q}(u_{k, q}+1),\forall k\in\ZZ_{>0}.$$
Then $u_{i, q}$ can be seen as a polynomial in $q$ with the leading term $q^{2^{i-1}}$. For any $p\in\ZZ_{\geq2}$, the set $\{u_{1, q}, \cdots, u_{p, q}\}$ satisfies the following properties:
\begin{align}
    \frac{q}{1+u_{1, q}}+ \cdots +\frac{q}{1+u_{p-1, q}}+\frac{q}{u_{p, q}}=1\label{sum1}\\
    \Pi_{i=1}^{p-1}\frac{1}{1+u_{i, q}}=\frac{q}{u_{p, q}}.\label{product}
\end{align}
In particular, $\{1+u_{i, 1}\}_{i\in\ZZ{>0}}$ is the Sylvester sequence.  

\begin{thm}\label{main thm}
    Fix two integers $d\geq3$ and $q>0$. Let $X$ be a $d$-dimensional $\frac{1}{q}$-lc toric Fano variety. Then $$\Vol(-K_X)\leq\frac{2u_{d, q}^2}{q^{d+1}},$$ and for $q\geq2$, the equality holds if and only if $$X\simeq\bP(\frac{2u_{d, q}}{1+u_{1, q}}, \cdots, \frac{2u_{d, q}}{1+u_{d-1, q}}, 1, 1).$$
\end{thm}
Note that the case when $q=1$ is already solved in \cite{AKN16} and \cite{BKN22} (see Theorem~\ref{bounds for d3}). 

\subsection{Upper bounds for dual normalized volumes of Fano polytopes}

To get the upper bound in Theorem~\ref{main thm}, it suffices to bound the normalized volume of the corresponding moment polytope, which is the dual of a $d$-dimensional $\frac{1}{q}$-lc Fano polytope. So we reduce Theorem~\ref{main thm} to Theorem~\ref{bounds for d3} and the following one.

\begin{thm}\label{dual vol bd for lattice polytope}
Fix two integers $d\geq3$ and $q\geq2$. Let $P$ be a $d$-dimensional $\frac{1}{q}$-lc Fano polytope. Then
    $$\Vol(P^*)\leq\frac{2u_{d, q}^2}{q^{d+1}},$$
and the equality holds if and only if $$P^*\simeq\conv(\frac{1+u_{1, q}}{q}e_{1}, \cdots, \frac{1+u_{d-1, q}}{q}e_{d-1}, \mp\frac{u_{d, q}}{q}e_{d}),$$ where $e_{i}$ are the standard basis of $\bR^d$.
\end{thm}

\subsection{Upper bounds for volumes of lattice simplices}

At the same time, we can get the following optimal upper bound for the volumes of lattice simplices $S\subset\bR^d$ such that $\inte(\frac{1}{q}S)\cap \ZZ^d=\{\0\}$.
\begin{thm}\label{vol bd for simplex}
Fix two integers $d\geq3$ and $q>0$. Let $S$ be a $d$-dimensional latttice simplex in $\bR^d$ such that $\inte(\frac{1}{q}S)\cap \ZZ^d=\{\0\}$. Then
    $$\vol_{\ZZ}(S)\leq\frac{2u_{d, q}^2}{d!q},$$
and the equality holds if $S\simeq\conv((1+u_{1, q})e_1-qe, \cdots, (1+u_{d-1, q})e_{d-1}-qe, -u_{d, q}e_d-qe, u_{d, q}e_d-qe)$, where $e_{i}$ are the standard basis of $\bR^d$ and $e=\Sigma_{i=1}^{d-1}e_i$.

\end{thm}
Note that the case when $q=1$ is already solved (see \cite[Theorem~2.2]{AKN16}).

\subsection{Organization of this paper}
In Section~2, we collect the necessary material from geometry of numbers and toric geometry, introducing the relevant polytopes, the way to bound the volume of a simplex (see Theorem~\ref{simplex vol bd}) and the decomposition strategy to deal with the dual volume of certain polytopes (see Theorem~\ref{decomp}); In Section~3, we introduce generalized Product-Sum inequalities (see Thereom~\ref{PS ineq}) for barycentric coordinates of the origin $\0$ with respect to an $\frac{1}{q}$-lc Fano simplex and use it to characterize when the products of barycentric coordinates attain the minimal ones (see Theorems~\ref{lower bd for bc d+1} and \ref{lower bd for bc d}); In Section~4, based on the results in Section~3, we apply Corollary~\ref{dual simplex vol bd} to get an upper bound for the dual normalized volume of any $d$-dimensional $\frac{1}{q}$-lc Fano simplex when $d\geq2$, and hence an upper bound for the anti-canonical volumes of $d$-dimensional $\frac{1}{q}$-lc toric Fano varieties of Picard number one (see Theorem~\ref{bd for the dual vol of s} and Theorem~\ref{bd volume p1}). We construct examples to show that these upper bounds are optimal when $d\geq3$ (see Example~\ref{ex1}), and we also show that under some conditions these examples are the unique ones which attain the optimal upper bounds; In Section~5, we use the decomposition strategy and integration method from \cite{BKN22} to find upper bounds for the dual normalized volumes of $\frac{1}{q}$-lc Fano polytopes which are not simplices; In Section~6, we prove Theorem~\ref{main thm}, Theorem~\ref{dual vol bd for lattice polytope} and Theorem~\ref{vol bd for simplex}.
\section{Preliminaries}

\subsection{Geometry of numbers}\label{GN}
We refer to \cite{GL87} for the basic knowledge in geometry of numbers. Fix a positive integer $d$. Let $\bR^d$ be the $d$-dimensional vector space and $\ZZ^d\subset\bR^d$ the lattice of rank $d$. A {\it unimodular transformation} is an affine bijection map $\phi: \bR^d\longrightarrow\bR^d $ such that $\phi(\ZZ^d)=\ZZ^d$. A {\it polytope} $P$ is a convex bounded subset of $\bR^d$ such that $P$ is the intersection of finitely many affine half-spaces. Two polytopes $P, Q$ are said to be {\it unimodularly equivalent}, denoted by $P\simeq Q$, if there exists a unimodular transformation $\phi$ such that $\phi(P)=Q$. For a polytope $P$, we use $\inte(P)$ and $\vex(P)$ to denote the interior of $P$ and the set of vertices of $P$ respectively. The {\it dual} of a polytope $P$ is defined as 
$$P^*:=\{m\in\bR^d|<m, v>+1\geq0, \forall v\in P\}.$$
\begin{definition}
    A polytope $P\subset\bR^d$ is called a {\it lattice polytope} if $\vex(P)\subset\ZZ^d$. We say $P\subset\bR^d$ is a {\it Fano polytope}  if it is a lattice polytope with $\0\in\inte(P)$ and each vertex $v\in\vex(P)$ is a primitive lattice point. A Fano polytope $P$ is called an {\it $\epsilon$-lc Fano polytope} for some $\epsilon>0$ if $\inte(\epsilon P)\cap\ZZ^d=\{\0\}$. Here $\0$ denotes the origin of $\bR^d$.
\end{definition}

\begin{definition}
    A {\it simplex} $S\subset\bR^d$ is a $d$-dimensional polytope with $d+1$ vertices. Let $S\subset\bR^d$ be a simplex with $\vex(S)=\{v_1, \cdots, v_{d+1}\}$. Given a point $x\in\bR^d$, the {\it barycentric coordinates} of $x$ with respect to $S$ are the uniquely determined numbers $\beta_1, \cdots, \beta_{d+1}$ such that 
    $$\Sigma_{i=1}^{d+1}\beta_i=1\quad \text{and}\quad \Sigma_{i=1}^{d+1}\beta_i v_i=x.$$ 
    Note that $x\in\inte(S)$ if and only if $\beta_i>0$ for all $i\in\{1, \cdots, d+1\}$.
\end{definition}

Given a $d$-dimensional polytope $P\subset\bR^d$, the {\it normalized volume} of $P$ is $\Vol(P)=d!\vol_\ZZ(P)$, where $\vol_\ZZ(P)$ is the (Euclidean) volume of $P$. For convenience, we call $\Vol(P^*)$ (resp. $\vol_\ZZ(P^*)$) the {\it dual normalized volume} (resp. {\it dual volume}) of $P$. 

In each dimension $d$, \cite[Lemma~5]{Pik01} gives an upper bound for the volumes of $d$-dimensional simplices containing interior lattice points in terms of the number of interior lattice points and the barycentric coordinates of an interior lattice point.

\begin{thm}\label{simplex vol bd}
Let $S$ be a $d$-dimensional simplex in $\bR^d$. Let $n=|\inte(S)\cap\ZZ^d|$ be the number of interior lattice points in $S$. Suppose $n\geq1$ and $x\in \inte(S)\cap\ZZ^d$ has barycentric coordinates $\beta_1\geq\cdots\geq\beta_{d+1}$. Then
    $$\vol_\ZZ(S)\leq\frac{n}{d!\beta_1\cdots\beta_d}.$$
\end{thm}

Note that for any $d$-dimensional lattice polytope $P\subset\bR^d$ with $\0\in\inte(P)$, we have $\inte(P^*)\cap\ZZ^d=\{\0\}$. In particular, for a Fano simplex $S$, the dual $S^*$ contains exactly one lattice point $\0$ in its interior. So we have the following corollary to bound the dual volume of a Fano simplex.

\begin{cor}\label{dual simplex vol bd}
Let $S$ be a $d$-dimensional Fano simplex in $\bR^d$. Suppose $\0\in \inte(S)$ has barycentric coordinates $\beta_1\geq\cdots\geq\beta_{d+1}$. Then
    $$\vol_\ZZ(S^*)\leq\frac{1}{d!\beta_1\cdots\beta_d}.$$
\end{cor}

\begin{proof}
    This follows from Theorem~\ref{simplex vol bd} and the fact that the barycentric coordinates of $\0$ with respect to $S^*$ and $S$ are the same. 
\end{proof}

To find the upper bound for anti-canonical volumes of certain toric Fano varieties, we need to bound the dual normalized volumes of the corresponding Fano polytopes from above (see Subsection~\ref{toric} below). Our first step is to bound the dual normalized volumes of certain Fano simplices. As indicated by Corollary~\ref{dual simplex vol bd}, it suffices to bound the product of the first $d$ barycentric coordinates of the origin $\0$ from below. For this purpose, we need to find out some effective inequalities for these barycentric coordinates. 

We will use the following Theorem to derive some important inequalities for barycentric coordinates of the origin $\0$ with respect to certain Fano simplices (see Theorem~\ref{PS ineq} below).

\begin{thm}[{cf. \cite[Theorem~2.2]{Aver12}}]\label{integral solution}
Let $A$ be a real $d\times d$ matrix such that $0<|\det(A)|<1$. Then there exists an integer point $z\in\ZZ^d\setminus\{\0\}$ such that $||A\cdot z||_{\infty}<1$, where $||\cdot||_{\infty}$ denotes the $l_{\infty}$-norm associated to the vector space $\bR^d$.
\end{thm}

To bound the dual normalized volumes of Fano polytopes which are not simplices, we will use the strategy in \cite{BKN22}. Therefore, we need to introduce the concept of minimal $\epsilon$-lc Fano polytopes.
\begin{definition}
Fix a positive real number $\epsilon>0$. We say a $d$-dimensional $\epsilon$-lc Fano polytope $P\subset\bR^d$ is {\it minimal} if for each $v\in\vex(P)$, the convex hull $\conv(P\cap\n\setminus\{v\})$ is not a $d$-dimensional $\epsilon$-lc Fano polytope in $\bR^d$.
\end{definition}

Note that by definition, every $\epsilon$-lc Fano simplex is minimal.

Given a $d$-dimensional $\epsilon$-lc Fano polytope $P\subset\bR^d$, we can find a $d$-dimensional minimal $\epsilon$-lc Fano polytope $P'\subset P$ by possibly removing several vertices of $P$. Note that $P'\subset P$ implies $\vol_{\ZZ}(P'^*)\geq\vol_{\ZZ}(P^*)$. So to find an upper bound for the dual volume of any $\epsilon$-lc Fano polytope, it is enough to bound the dual volumes for minimal $\epsilon$-lc Fano polytopes from above, in which case we will use the decomposition of a minimal $\epsilon$-lc Fano polytope into lower dimensional $\epsilon$-lc Fano simplices:
\begin{thm}[{cf. \cite[Proposition~3.2]{Kas10} and \cite[Corollary~2.3]{BKN22}}]\label{decomp}
    Let $P\subset\bR^d$ be a $d$-dimensional minimal $\epsilon$-lc Fano polytope. If $P$ is not a simplex, then there exist lower dimensional $\epsilon$-lc Fano simplices $S_1, \cdots, S_t (2\leq t\leq d)$ such that $P=\conv(S_1\cup\cdots\cup S_t)$, where $\vex(S_i)\subset\vex(P) (1\leq i\leq t)$. For $1\leq i\leq t$, denote $d_i:=\dim S_i$ and $r_i:=|\vex(S_i)\cap\vex(P_{i-1})|$, where $P_{i-1}=\conv(S_1\cup\cdots\cup S_{i-1})$ for $2\leq i\leq t$ and $P_{0}$ is the empty set. We have the following:
    \begin{align}
        d_1+\cdots+d_t=d+\Sigma_{i=1}^{t}r_i;\label{sigma di}\\
        r_i<d_i\leq d-t+1, \forall 1\leq i\leq t;\label{bd di}\\
        |\vex(P)|=d+t.\label{number of vert}
    \end{align}
\end{thm}
\begin{remark}\label{rem1}
    In the proof of \cite[Proposition~3.2]{Kas10}, only the property that Fano polytopes have $\0$ in the interior really makes a difference. So the decomposition strategy also works for minimal $\epsilon$-lc Fano polytopes.
\end{remark}

\subsection{Anti-canonical volumes of toric Fano varieties}\label{toric}
We refer to \cite{CLS11} and \cite{Fulton} for the basic knowledge in toric geometry. 

Let $X$ be a $d$-dimensional $\epsilon$-lc toric Fano variety associated to the fan $\Delta\subset\bR^d$. Denote by $P_X$ the convex hull $\conv(e_i, e_i\in\Delta(1))$, where $\Delta(1)$ is the set of primitive generators of one-dimensional rays of $\Delta$. Then 
$$\mld(X)=\inf\{t>0| tP_X\cap\ZZ^d\supsetneq\{\0\}\}\geq\epsilon,$$ which implies 
$P_X$ is an $\epsilon$-lc Fano polytope. We refer to \cite[Subsections~6.1 and 6.2]{Ambro16} for the details. The dual polytope of $P_X$ is just the moment polytope associated to $-K_X$:
$$P_{X}^{*}=\{m\in\bR^d|<m, e_i>+1\geq0, \forall e_i\in\Delta(1)\}.$$
According to \cite[Theorem~13.4.3]{CLS11}, we have 
$$\Vol(-K_X)=\Vol(P_{X}^{*})=d!\vol_\ZZ(P_{X}^{*}).$$
The following are known results for the optimal upper bounds for anti-canonical volumes of canonical toric Fano varieties in dimension $d\geq3$.

%\begin{thm}[{cf. \cite[Theorem~1.3]{Jiang13}}]\label{bounds for surface}
    %Let $X$ be an $\epsilon$-lc del Pezzo surface. Then 
    %$$\Vol(-K_X)\leq\max\{9, \lfloor2/\epsilon\rfloor+4+\frac{4}{\lfloor2/\epsilon\rfloor}\}$$
%\end{thm}

\begin{thm}[{cf. \cite[Theorem~4.6]{Kas10}, \cite[Theorem~2.11]{AKN16} and \cite[Corollary~1.3]{BKN22}}]\label{bounds for d3}
    Fix a positive integer $d\geq3$. Let $X$ be a $d$-dimensional canonical toric Fano variety. Then 
    $$\Vol(-K_X)\leq 2u_{d, 1}^{2}.$$
   For $d=3$, the equality holds if and only if $X\simeq\bP(6, 4, 1, 1)$ or $\bP(3, 1, 1, 1)$; for $d\geq4$, the equality holds if and only if $X\simeq\bP(\frac{2u_{d, 1}}{1+u_{1, 1}}, \cdots, \frac{2u_{d, 1}}{1+u_{d-1, 1}}, 1, 1)$. 
\end{thm}

\section{Inequalities for barycentric coordinates}

Fix two integers $d\geq2$ and $q\geq1$. Let $S$ be a $d$-dimensional $\frac{1}{q}$-lc Fano simplex. Suppose $\beta_1\geq\beta_2\geq...\geq\beta_{d+1}$ are the barycentric coordinates of $\0\in \inte(S)$ with respect to the vertices $v_1, v_2,...,v_{d+1}$. In this section, we will give lower bounds for  $\Pi_{i=1}^{d+1}\beta_i$ and $\Pi_{i=1}^{d}\beta_i$, which are the main ingredients for the proof of the main results.

By generalizing  the proof of \cite[Theorem~1.1]{Aver12} to the situation considered here, we can get generalized Product-Sum inequalities for barycentric coordinates.
\begin{thm}\label{PS ineq}
Keep the same notation as above. Then the barycentric coordinates $\beta_1, \beta_2,...,\beta_{d+1}$ satisfy the generalized Product-Sum inequalities, that is, for each $t\in\{1, 2,..., d\}$, we have
    $$\Pi_{i=1}^{t}\beta_i\leq q^t\Sigma_{j=t+1}^{d+1}\beta_j.$$
\end{thm}

\begin{proof}
    Fix $t\in\{1, 2,..., d\}$. Suppose to the contrary that $\Pi_{i=1}^{t}\beta_i>q^t\Sigma_{j=t+1}^{d+1}\beta_j.$ We will derive a contradiction by showing that there is an interior lattice point in $\frac{1}{q}S$ different from $\0$.
    Consider the $(t+1)\times(t+1)$ matrix $A:=$
    \[
\begin{bmatrix}
  \frac{q}{\beta_1} & 0 & \cdots & 0 & -q \\
  0 & \frac{q}{\beta_2} & \cdots & 0 & -q \\
  \vdots & \vdots & \ddots & \vdots & \vdots \\
  0 & 0 & \cdots & \frac{q}{\beta_t} & -q \\
  -1 & -1 & \cdots & -1 & 1
\end{bmatrix}
\]
Then $0<\det A=\frac{q^t(1-\Sigma_{i=1}^{t}\beta_i)}{\Pi_{i=1}^{t}\beta_i}=\frac{q^t\Sigma_{j=t+1}^{d+1}\beta_j}{\Pi_{i=1}^{t}\beta_i}<1$. By Theorem~\ref{integral solution}, there exists a point $(m_1, m_2,\cdots,m_t, m)\in\ZZ^{t+1}\setminus\{\0\}$ such that 
\begin{align}\label{ineq}
    \max\{|\frac{m_1 q}{\beta_1}-mq|,\cdots,|\frac{m_t q}{\beta_t}-mq|, |\Sigma_{i=1}^{t}m_i-m|\}<1.
\end{align}

Note that $m$ cannot be $0$, otherwise, as indicated by \eqref{ineq}, all $m_i$ would be $0$. Possibly by replacing $m_1, m_2,\cdots,m_t, m$ with $-m_1, -m_2,\cdots,-m_t, -m$, we can assume that $m>0$. By construction $|\Sigma_{i=1}^{t}m_i-m|<1$ and $m_1, m_2,\cdots,m_t, m$ are all integers, we have $\Sigma_{i=1}^{t}m_i=m$. 

Set $v=-q\Sigma_{i=1}^{t}m_i v_i$. Then $v=(mq+1)\0+v=\Sigma_{i=1}^{t}((mq+1)\beta_i-m_i q) v_i+\Sigma_{j=t+1}^{d+1}(mq+1)\beta_j v_j$. For each $i\in\{1,\cdots,t\}$, inequality \eqref{ineq} implies $\frac{m_i q}{\beta_i}-mq<1$, which in turn implies $(mq+1)\beta_i-m_i q>0$. On the other hand, $\Sigma_{i=1}^{t}((mq+1)\beta_i-m_i q)+\Sigma_{j=t+1}^{d+1}(mq+1)\beta_j=1$. Thus we have $v\in\inte(S)$, which implies that $\0\neq\frac{1}{q}v\in\inte(\frac{1}{q}S)\cap\ZZ^d$, contradicting the assumption that $S$ is a $\frac{1}{q}$-lc Fano simplex.
\end{proof}

\begin{remark}\label{rem2}
    From the proof we can see that the conclusions also hold for a $d$-dimensional latttice simplex $S$ in $\bR^d$ such that $\inte(\frac{1}{q}S)\cap \ZZ^d=\{\0\}$.
\end{remark}

As indicated by Theorem~\ref{PS ineq}, in order to find lower bounds for $\Pi_{i=1}^{d+1}\beta_i$ and $\Pi_{i=1}^{d}\beta_i$, we need to consider the set of $(d+1)$-tuples $(x_1, x_2, \cdots, x_{d+1})\in \bR^{d+1}$  such that the following conditions hold:
\begin{align}
x_1+\cdots+x_{d+1}&{}=1,\label{relation1} \\
1\geq x_1\geq x_2\geq\cdots\geq x_{d+1}&{}\geq0,\label{relation2}\\
\forall t\in\{1, 2, \cdots, d\}, \Pi_{i=1}^{t}x_i&{}\leq q^t\Sigma_{j=t+1}^{d+1}x_j.\label{relation3}
\end{align}
We denote this set by $X(d, q)$. Then $(\beta_1, \cdots, \beta_{d+1})\in X(d, q)$ by Theorem~\ref{PS ineq}. For convenience, let $\ps(x)_t (t\in\{1, 2, \cdots, d\})$ be the $t$-th Product-Sum inequality $\Pi_{i=1}^{t}x_i\leq q^t\Sigma_{j=t+1}^{d+1}x_j$ with respect to the positive integer $q$ and the $(d+1)$-tuple $x=(x_1, x_2, \cdots, x_{d+1})\in\bR^{d+1}$. By definition, for any $x\in X(d, q)$, $\ps(x)_t (t\in\{1, 2, \cdots, d\})$ holds. And $x\in X(d, q)$ also satisfies the following basic properties (a generalized version of \cite[Lemma~4.1]{AKN16}):

\begin{lem}\label{eq}
Let $x=(x_1, x_2, \cdots, x_{d+1})\in X(d, q)$. Assume that $d\geq2$. Then the following conditions hold:
\begin{enumerate}
    \item $x_1<1$ and $x_{d+1}>0$;
    \item If the equality in $\ps(x)_t$ holds for some $t\in\{1,\cdots,d\}$, then $x_t>x_{t+1}$;
    \item Fix $k\in\{1,\cdots, d\}$, if the equality in $\ps(x)_t$ holds for any $t\in\{1,\cdots,k\}$, then $x_t=\frac{q}{1+u_{t, q}}$ for each $t\in\{1,\cdots,k\}$.
\end{enumerate}
\end{lem}
\begin{proof}
    $(1)$ If $x_1=1$, then \eqref{relation1} and \eqref{relation2} imply $x_i=0$ for all $i\in\{2,\cdots,d+1\}$, which contradicts $\ps(x)_1$. If $x_{d+1}=0$, then \eqref{relation2} and $\ps(x)_d$ imply that $x_{d}=0$. Inductively we can apply \eqref{relation2} and \eqref{relation3} to deduce that $x_i=0$ for all $i\in\{1,\cdots,d+1\}$, which is a contradiction.

    $(2)$ By assumption, we have 
    $$x_t\geq\Pi_{i=1}^{t}x_i=q^t\Sigma_{j=t+1}^{d+1}x_j\geq x_{t+1},$$
    where at least one of the two inequalities is strict because $d\geq2$ and $0<x_i<1$ for all $i\in\{1,\cdots,d+1\}$. Hence $x_t>x_{t+1}$.
    
    $(3)$ For $t=1$, $x_1=q\Sigma_{j=2}^{d+1}x_j$ implies $x_1=q(1-x_1)$, and hence $x_1=\frac{q}{1+u_{1, q}}$. Assume $k\geq2$ and for some $1\leq t\leq k-1$, we have $x_i=\frac{q}{1+u_{i, q}}$ for each $i\in\{1, \cdots, t\}$. Then $\Pi_{i=1}^{t+1}x_i=q^{t+1}\Sigma_{j=t+2}^{d+1}x_j$ implies $(\Pi_{i=1}^{t}\frac{q}{1+u_{i, q}})x_{t+1}=q^{t+1}(1-\Sigma_{i=1}^{t}\frac{q}{1+u_{i, q}}-x_{t+1})$. Applying \eqref{sum1} and \eqref{product} with $p=t+1$, we can get $x_{t+1}=\frac{q}{1+u_{t+1, q}}.$ So we conclude the proof by induction.
\end{proof}

Now we come to study the lower bounds of $\Pi_{i=1}^{d+1}\beta_i$ and $\Pi_{i=1}^{d}\beta_i$.

\begin{lem}\label{d+1prod}
    Let $f_{d+1}: \bR^{d+1}\longrightarrow \bR$ be the continuous function defined by sending $x=(x_1, \cdots, x_{d+1})$ to $f_{d+1}(x)=\Pi_{i=1}^{d+1}x_i$. Let $b$ be the $(d+1)$-tuple $(b_1, \cdots, b_{d+1})\in X(d, q)$ such that $f_{d+1}(b)=\min\{f_{d+1}(x)|x\in X(d, q)\}$. Then there exists an integer $k\in\{0, 1, \cdots, d\}$ such that 
    \begin{enumerate}
        \item [(a)] $b_l=b_{l+1}$ for $k+1\leq l\leq d$, and
        \item [(b)] the equality in $\ps(b)_l$ holds for $1\leq l\leq k$.
    \end{enumerate}
\end{lem}

\begin{proof}
    Set $b_0:=1$, and choose $k\in\{0, 1, \cdots, d\}$ such that
    \begin{align}
        b_k>b_{k+1}=\cdots=b_{d+1}.\label{k}
    \end{align}

    If $k=0$, then there is nothing to prove. 

    So from now on, we may assume that $k\geq1$.

    \begin{claim}\label{strict order}
        If $k\geq2$, then $b_i>b_{i+1}$ for any $1\leq i\leq k-1$.
    \end{claim}
\begin{proof}
    If the conclusion does not hold, then we may find $1\leq i_1<i_2\leq k$ such that $b_{i_1-1}>b_{i_1}=\cdots=b_{i_2}>b_{i_2+1}$. Take a real number $\delta>0$ and set $b':=(b_1, \cdots, b_{i_1}+\delta, \cdots, b_{i_2}-\delta, \cdots, b_{d+1})$, i.e., $b'$ is attained from $b$ by adding $\delta$ to $b_{i_1}$ and subtracting $b_{i_2}$ by $\delta$. Then $b'$ satisfies condition \eqref{relation1}. And condition \eqref{relation2} holds for $b'$ if $\delta$ is sufficiently small. Furthermore, $b'$ aslo satisfies condition \eqref{relation3} for a sufficiently small $\delta$: in fact, for $1\leq l<i_1$, $\ps(b')_l$ holds because it is the same as $\ps(b)_l$; for $i_1\leq l<i_2$, note that by Lemma~\ref{eq}~$(2)$, $\ps(b)_l$ is strict, so $\ps(b')_l$ holds if we take $\delta>0$ small enough; for $i_2\leq l\leq d$, the left product part of $\ps(b')_l$ is strict less than that of $\ps(b)_l$ by the fact that $(b_{i_1}+\delta)(b_{i_2}-\delta)<b_{i_1}b_{i_2}$, and the right sum parts of them are the same, so $\ps(b')_l$ holds. Hence for a sufficiently small real number $\delta>0$, we have $b'\in X(d, q)$. But $(b_{i_1}+\delta)(b_{i_2}-\delta)<b_{i_1}b_{i_2}$ implies $f_{d+1}(b')<f_{d+1}(b)$, which contradicts the choice of $b$. So we finish the proof of the claim. 
\end{proof}  

\begin{claim}
    If  $k\geq2$, then the equality in $\ps(b)_l$ holds for any $1\leq l\leq k-1$. 
\end{claim}

\begin{proof}
    Suppose $\ps(b)_l$ is strict for some $1\leq l\leq k-1$. Set $b^*:=(b_1, \cdots, b_l+\delta, b_{l+1}-\delta, \cdots, b_{d+1})$ for some $\delta>0$. Then $b^*$ satisfies condition \eqref{relation1}. By claim~\ref{strict order} and \eqref{k}, we have $b_{l-1}>b_l$ and $b_{l+1}>b_{l+2}$, which imply condition \eqref{relation2} holds for $b^*$ if $\delta>0$ is small enough. By comparing $\ps(b^*)_t$ and $\ps(b)_t$ for each $1\leq t \leq d$, we can see that $\ps(b^*)_t$ holds if $\delta>0$ is small enough. Hence for a sufficiently small real number $\delta>0$, we have $b^*\in X(d, q)$. But $f_{d+1}(b^*)<f_{d+1}(b)$ by construction, contradicting the choice of $b$.
\end{proof}
To finish the proof, we need to show that the equality in $\ps(b)_k$ holds. Suppose to the contrary that $\Pi_{i=1}^{k}b_i<q^t\Sigma_{j=k+1}^{d+1}b_j$. For a real number $\delta$, set $$b(\delta):=(b_1, \cdots, b_k-(d+1-k)\delta, b_{k+1}+\delta, \cdots, b_{d+1}+\delta)$$ and $$g(\delta):=\frac{f_{d+1}(b(\delta))}{\Pi_{i=1}^{k-1}b_i}=(b_{d+1}+\delta)^{d+1-k}\cdot(b_k-(d+1-k)\delta).$$ Note that $\ps(b)_l$ is strict for each $k\leq l\leq d$ by Lemma~\ref{eq} $(2)$ and the assumptions. Arguing similarly as in the proof of the above two claims, we can have $b(\delta)\in X(d, q)$ as $\delta$ varies in a small open neighbourhood of $0$. By construction $g(\delta)$ has a local minimal value at $\delta=0$. So the derivative value of $g(\delta)$ at 0 must be $0$, which implies $b_{k}=b_{d+1}$, a contradiction.
    
\end{proof}

\begin{lem}\label{dprod}
    Let $f_{d}: \bR^{d+1}\longrightarrow \bR$ be the continuous function defined by sending $x=(x_1, \cdots, x_{d+1})$ to $f_{d}(x)=\Pi_{i=1}^{d}x_i$. Let $a$ be the $(d+1)$-tuple $(a_1, \cdots, a_{d+1})\in X(d, q)$ such that $f_{d}(a)=\min\{f_{d}(x)|x\in X(d, q)\}$. Then there exists an integer $k\in\{0, 1, \cdots, d-1\}$ such that 
    \begin{enumerate}
        \item [(a)] $a_l=a_{l+1}$ for $k+1\leq l\leq d$, and
        \item [(b)] the equality in $\ps(a)_l$ holds for $1\leq l\leq k$.
    \end{enumerate}
\end{lem}

\begin{proof}
    First we claim that $a_{d}=a_{d+1}$ always holds. Suppose $a_{d}>a_{d+1}$. Then one can take a small $\delta>0$ such that $a_{d}-\delta>a_{d+1}+\delta$. Set $a^*=(a_1, \cdots, a_{d-1}, a_{d}-\delta, a_{d+1}+\delta)$. We have $a^*\in X(d, q)$ and $f_{d}(a^*)<f_{d}(a)$, contradicting the choice of $a$.

    We can proceed the proof as in the proof of Lemma~\ref{d+1prod} to deduce that there exists an integer $k\in\{0, 1, \cdots, d-1\}$ such that $$a_1>\cdots>a_k>a_{k+1}=\cdots=a_{d+1}$$ and $\text{the equality in } \ps(a)_l \text{ holds for } 1\leq l\leq k-1.$
    So it suffices to prove that the equality in $\ps(a)_k$ holds. Suppose to the contrary that $\ps(a)_k$ is strict. For a real number $\delta$, set $$a(\delta):=(a_1, \cdots, a_k-(d+1-k)\delta, a_{k+1}+\delta, \cdots, a_{d+1}+\delta)$$ and $$h(\delta):=\frac{f_{d}(a(\delta))}{\Pi_{i=1}^{k-1}a_i}=(a_{d+1}+\delta)^{d-k}\cdot(a_k-(d+1-k)\delta).$$ Arguing similarly as before, we can have $a(\delta)\in X(d, q)$ as $\delta$ varies in a small open neighbourhood of $0$. Consider the derivative of $h(\delta)$:
    $$h'(\delta)=(d-k)(a_{d+1}+\delta)^{d-1-k}(a_k-(d+1-k)\delta)-(d+1-k)(a_{d+1}+\delta)^{d-k}.$$
    By construction $h(\delta)$ has a local minimal value at $\delta=0$, we must have $h'(0)=0$, which implies $a_{d+1}(d+1-k)=a_k(d-k)$. Consider the double derivative value of $h(\delta)$ at $\delta=0$:
    $$ h''(0)=(d-k)(d-1-k)a_{d+1}^{d-2-k}a_k-2(d-k)(d+1-k)a_{d+1}^{d-1-k}.
%\begin{cases}
%(d-k)(d-1-k)a_{d+1}^{d-2-k}a_k-2(d-k)(d+1-k)a_{d+1}^{d-1-k} & \text{if } k<d-1;\\
%-4  & \text{if } k=d-1.
%\end{cases}
$$
Since $a_{d+1}(d+1-k)=a_k(d-k)$, we have $h''(0)=-(d+1-k)^2\cdot a_{d+1}^{d-1-k}<0$, and hence $h(\delta)$ has a local maximal value at $\delta=0$, a contradiction.  
\end{proof}

As indicated by Lemma~\ref{d+1prod} (resp. Lemma~\ref{dprod}) and Lemma~\ref{eq} $(3)$, in order to find the lower bound of $\Pi_{i=1}^{d+1}\beta_i$ (resp. $\Pi_{i=1}^{d}\beta_i$), we need to consider the $(d+1)$-tuple
$$y[q](l):=(\frac{q}{1+u_{1, q}}, \cdots, \frac{q}{1+u_{l-1, q}}, \frac{q}{(d+2-l)u_{l, q}}, \cdots, \frac{q}{(d+2-l)u_{l, q}})$$ 
for each $l\in\{1, \cdots, d+1\}$. Note that by definition, we have  
\begin{align*}
         y[q](1)&{}=(\frac{1}{d+1}, \cdots, \frac{1}{d+1}),\\
         y[q](d)&{}=(\frac{q}{1+u_{1, q}}, \cdots, \frac{q}{1+u_{d-1, q}}, \frac{q}{2u_{d, q}},  \frac{q}{2u_{d, q}}), \\
         y[q](d+1)&{}=(\frac{q}{1+u_{1, q}}, \cdots, \frac{q}{1+u_{d, q}}, \frac{q}{u_{d+1, q}}).
     \end{align*}
It is easy to see that for each $l\in\{1, \cdots, d+1\}$, $y[q](l)\in X(d, q)$. 

\begin{thm}\label{lower bd for bc d+1}
Fix two integers $d\geq2$ and $q\geq1$. Let $S$ be a $d$-dimensional $\frac{1}{q}$-lc Fano simplex. Suppose $\beta_1\geq\beta_2\geq...\geq\beta_{d+1}$ are the barycentric coordinates of $\0\in \inte(S)$ with respect to $S$. Let $f_{d+1}: \bR^{d+1}\longrightarrow \bR$ be the continuous function defined by sending $x=(x_1, \cdots, x_{d+1})$ to $f_{d+1}(x)=\Pi_{i=1}^{d+1}x_i$. Then $$\Pi_{i=1}^{d+1}\beta_i\geq f_{d+1}(y[q](d+1))=\frac{q^{d+2}}{u_{d+1, q}^2},$$
and the equality holds if and only if $$(\beta_1, \cdots, \beta_{d+1})=(\frac{q}{1+u_{1, q}}, \cdots, \frac{q}{1+u_{d, q}}, \frac{q}{u_{d+1, q}}).$$
\end{thm}

\begin{proof}

    According to Lemma~\ref{d+1prod} and Lemma~\ref{eq} $(3)$, we have $$\min\{f_{d+1}(x)| x\in X(d, q)\}=\min\{f_{d+1}(y[q](l))| l\in\{1, \cdots, d+1\}\}.$$
    So it suffices to show that 
    \begin{align}
        f_{d+1}(y[q](d+1))<f_{d+1}(y[q](l)),\forall l\in\{1, \cdots, d\}.\label{e1}
    \end{align}
    If $d=2$, then we have    
    \begin{align*}
         y[q](1)&{}=(\frac{1}{3}, \frac{1}{3}, \frac{1}{3}),\\
         y[q](2)&{}=(\frac{q}{1+q}, \frac{1}{2(1+q)}, \frac{1}{2(1+q)}), \\
         y[q](3)&{}=(\frac{q}{1+q}, \frac{q}{1+q+q^2}, \frac{q}{(q+q^2)(1+q+q^2)}).
     \end{align*}
    One can check directly that \eqref{e1} holds.
    So we may assume that $d\geq3$.

    In case $q=1$, \eqref{e1} follows from \cite[Theorem~4.3 $(c)$]{AKN16}.
    So in the next discussion, we always assume that $q\geq2$.

    We claim that for $2\leq l\leq d$, $f_{d+1}(y[q](l))>f_{d+1}(y[q](l+1))$.

    By definition, $f_{d+1}(y[q](l))>f_{d+1}(y[q](l+1))$ is equivalent to 
    \begin{align}
        &{}(\Pi_{i=1}^{l-1}\frac{q}{1+u_{i, q}})\cdot(\frac{q}{(d+2-l)u_{l, q}})^{d+2-l}\notag \\
        &{}>(\Pi_{i=1}^{l}\frac{q}{1+u_{i, q}})\cdot(\frac{q}{(d+1-l)u_{l+1, q}})^{d+1-l}.\label{ineq1}
    \end{align}
    Since $u_{l+1, q}=u_{l, q}(u_{l, q}+1)$, \eqref{product} holds with $p=l$ and $p=l+1$, \eqref{ineq1} is equivalent to
    $$(d+2-l)^{d+2-l}\cdot u_{l, q}<(d+1-l)^{d+1-l}\cdot (1+u_{l, q})^{d+2-l}.$$
    Since $(1+\frac{1}{d+1-l})^{d+1-l}<e$, where $e$ is the Euler number, the above inequality holds if 
    \begin{align}
        (d+2-l)e<(1+u_{l, q})^{d+1-l}.\label{ineq2}
    \end{align}
    Consider the function $g(t):=(1+u_{l, q})^{t+1}-e(t+2), t\in [0, +\infty)$. Then $g'(t)=(1+u_{l, q})^{t+1}\ln (1+u_{l, q})-e$. We have $g'(t)>0, \forall t\in [0, +\infty)$ by noting that $u_{l, q}\geq q(1+q)\geq2$. Hence $g(d-l)\geq g(0)=(1+u_{l, q})-2e\geq 1+q(1+q)-2e>0.$ Then \eqref{ineq2} holds for $2\leq l\leq d$, which in turn implies \eqref{ineq1} holds.

    Finally, by the inequality for the geometric and arithmetic means, we have $f_{d+1}(y[q](1))=(\frac{1}{d+1})^{d+1}>f_{d+1}(y[q](d+1))$. So for all $1\leq l\leq d$, $f_{d+1}(y[q](l))>f_{d+1}(y[q](d+1))$.
\end{proof}

\begin{thm}\label{lower bd for bc d}
Fix two integers $d\geq2$ and $q\geq1$. Let $S$ be a $d$-dimensional $\frac{1}{q}$-lc Fano simplex. Suppose $\beta_1\geq\beta_2\geq...\geq\beta_{d+1}$ are the barycentric coordinates of $\0\in \inte(S)$ with respect to $S$. Let $f_{d}: \bR^{d+1}\longrightarrow \bR$ be the continuous function defined by sending $x=(x_1, \cdots, x_{d+1})$ to $f_{d}(x)=\Pi_{i=1}^{d}x_i$. Then $$\Pi_{i=1}^{d}\beta_i\geq f_{d}(y[q](d))=\frac{q^{d+1}}{2u_{d, q}^2},$$
except for the case $d=2, q=1$. In the exceptional case $d=2, q=1$, $\Pi_{i=1}^{d}\beta_i\geq f_{2}(y[q](1))=\frac{1}{9}.$ Furthermore, for $d=3, q=1$ the equality holds if and only if $(\beta_1, \beta_2, \beta_3, \beta_{4})=(\frac{1}{2}, \frac{1}{6}, \frac{1}{6}, \frac{1}{6})$ or $(\frac{1}{2}, \frac{1}{3}, \frac{1}{12}, \frac{1}{12})$; for $d=3, q\geq2$ or $d\geq4$, the equality holds if and only if $$(\beta_1, \cdots, \beta_{d+1})=(\frac{q}{1+u_{1, q}}, \cdots, \frac{q}{1+u_{d-1, q}}, \frac{q}{2u_{d, q}}, \frac{q}{2u_{d, q}}).$$
\end{thm}

\begin{proof}
According to Lemma~\ref{dprod} and Lemma~\ref{eq} $(3)$, we have $$\min\{f_{d}(x)| x\in X(d, q)\}=\min\{f_d(y[q](l))| l\in\{1, \cdots, d\}\}.$$

If $d=2$, then we have $f_{2}(y[q](1))=\frac{1}{9}$ and $f_{2}(y[q](2))=\frac{q}{2(1+q)^2}$. So In case $q=1$, we have $\min\{f_{2}(x)| x\in X(2, q)\}=\frac{1}{9}.$ And for $q\geq 2$, we have $\min\{f_{2}(x)| x\in X(2, q)\}=f_{2}(y[q](2)).$ 

In case $q=1$, the conclusions for $d\geq3$ follow from \cite[Theorem~4.3 $(e)$]{AKN16}.
So in the following, we assume that $d\geq3$ and $q\geq2$.

We claim that for $2\leq l\leq d-1$, $f_{d}(y[q](l))>f_{d}(y[q](l+1))$.

    By definition, $f_{d}(y[q](l))>f_{d}(y[q](l+1))$ is equivalent to 
    \begin{align}
        &{}(\Pi_{i=1}^{l-1}\frac{q}{1+u_{i, q}})\cdot(\frac{q}{(d+2-l)u_{l, q}})^{d+1-l}\notag \\
        &{}>(\Pi_{i=1}^{l}\frac{q}{1+u_{i, q}})\cdot(\frac{q}{(d+1-l)u_{l+1, q}})^{d-l}.\label{ineq1+}
    \end{align}
    Since $u_{l+1, q}=u_{l, q}(u_{l, q}+1)$, \eqref{product} holds with $p=l$ and $p=l+1$, \eqref{ineq1+} is equivalent to
    $$(d+2-l)^{d+1-l}\cdot u_{l, q}<(d+1-l)^{d-l}\cdot (1+u_{l, q})^{d+1-l}.$$
    Since $(1+\frac{1}{d+1-l})^{d+1-l}<e$, the above inequality holds if 
    \begin{align}
        (d+1-l)e<(1+u_{l, q})^{d-l}.\label{ineq2+}
    \end{align}
    Consider the function $g(t):=(1+u_{l, q})^{t+1}-e(t+2), t\in [0, +\infty)$. Then by the same argument as in the proof of Theorem~\ref{lower bd for bc d+1}, we have $g(d-1-l)\geq g(0)=(1+u_{l, q})-2e\geq 1+q(1+q)-2e>0.$ Then \eqref{ineq2+} holds for $2\leq l\leq d-1$, which in turn implies \eqref{ineq1+} holds.

    It remains to prove that $f_{d}(y[q](1))>f_{d}(y[q](d))$ for $d\geq3$ and $q\geq2$. $f_{d}(y[q](1))>f_{d}(y[q](d))$ is equivalent to
    \begin{align}
        (d+1)^d\cdot q^{d+1}<2u_{d, q}^2.\label{ineq3}
    \end{align}
   If $d=3$, then \eqref{ineq3} holds by noting that $u_{3, q}^2>u_{2, q}^4=q^4(1+q)^4$. We proceed by induction. Assume \eqref{ineq3} holds for some $d\geq3$, we need to prove \eqref{ineq3} holds when we change $d$ to $d+1$, i.e., to show the following inequality holds:
     \begin{align}
        (d+2)^{d+1}\cdot q^{d+2}<2u_{d+1, q}^2.\label{ineq3+}
    \end{align}
    Since $(d+1)^d\cdot q^{d+1}<2u_{d, q}^2$ and $u_{d+1, q}=u_{d, q}(1+u_{d, q})$, \eqref{ineq3+} holds if $$(d+2)^{d+1}q<(d+1)^d(1+u_{d, q})^2.$$ 
    Since $(1+\frac{1}{d+1})^{d+1}<e$, the above inequality holds if $eq(d+1)<(1+u_{d, q})^2$. This is easy to see by noting that $u_{d, q}>q^{2^{d-1}}.$ 
\end{proof}

\section{Bounding the dual normalized volumes of Fano simplices}
Fix two integers $d\geq2$ and $q\geq1$. In this section, we give the upper bound for the dual normalized volumes of $d$-dimensional $\frac{1}{q}$-lc Fano simplices, which corresponds to the upper bound for the 
anti-canonical volumes of $d$-dimensional $\frac{1}{q}$-lc toric Fano varieties with Picard number one. We construct examples to show that the upper bounds are optimal in each dimension $d\geq3$.

\begin{thm}\label{bd for the dual vol of s}
 Fix two integers $d\geq2$ and $q\geq1$. Then for any $d$-dimensional $\frac{1}{q}$-lc Fano simplex $S\subset\bR^d$, we have
\begin{enumerate}
  \item [(a)] if $d=2$, then $\Vol(S^*)\leq9$ provided $q=1$, and $\Vol(S^*)\leq\frac{2u_{2, q}^{2}}{q^3}$ provided $q\geq2$;
  \item [(b)] if $d\geq3$, then $\Vol(S^*)\leq\frac{2u_{d, q}^{2}}{q^{d+1}}$, and
  \item [(c)] for $q\geq2$, the equality holds if and only if  $$S^*\simeq\conv(\frac{1+u_{1, q}}{q}e_{1}, \cdots, \frac{1+u_{d-1, q}}{q}e_{d-1}, \mp\frac{u_{d, q}}{q}e_{d}),$$
  where $e_{i}$ are the standard basis of $\bR^d$.
\end{enumerate}
\end{thm}

\begin{proof}
Suppose $\beta_1, \cdots, \beta_{d+1}$ are the barycentric coordinates of $\0$ with respect to $S$, sorted decreasingly as $\beta_1\geq\cdots\geq\beta_{d+1}$. Then $\beta_1, \cdots, \beta_{d+1}$ are the barycentric coordinates of $\0$ with respect to the dual $S^*$ of $S$. By Corollary~\ref{dual simplex vol bd}, we have $$\Vol(S^*)\leq\frac{1}{\beta_1\cdots\beta_d}.$$
So the conclusions $(a)$ and $(b)$ follow from Theorem~\ref{lower bd for bc d}. The proof of conclusion $(c)$ will be presented at the end of this section.
\end{proof}

\begin{thm}\label{bd volume p1}
Fix two integers $d\geq2$ and $q\geq1$. Then for any $d$-dimensional $\frac{1}{q}$-lc toric Fano variety $X$ with Picard number one, we have
\begin{enumerate}
\item [(a)] if $d=2$, then $(-K_X)^2\leq9$ provided $q=1$, and $(-K_X)^2\leq\frac{2u_{2, q}^{2}}{q^3}$ provided $q\geq2$;
\item [(b)] if $d\geq3$, then $(-K_X)^d\leq\frac{2u_{d, q}^{2}}{q^{d+1}}$, and
\item [(c)] for $q\geq2$ the equality holds if and only if $$X\simeq\bP(\frac{2u_{d, q}}{1+u_{1, q}}, \cdots, \frac{2u_{d, q}}{1+u_{d-1, q}}, 1, 1).$$
\end{enumerate}
\end{thm}
\begin{proof}
By assumption, the associated fan of $X$ is spanned by the faces of a $d$-dimensional $\frac{1}{q}$-lc Fano simplex $S$. 
Then by \cite[Theorem~13.4.3]{CLS11}, we have 
$$(-K_X)^d=\Vol(S^*).$$
So the conclusions follow from Theorem~\ref{bd for the dual vol of s} along with the following Example~\ref{ex1} and the proof of Theorem~\ref{bd for the dual vol of s}$(c)$.
\end{proof}

\begin{example}\label{ex1}
Let $e_1, \cdots, e_d$ be the standard basis of $\bR^d$. Set $e=\Sigma_{i=1}^{d-1}e_i$. Let $S$ be the convex hull 
of the set of $d+1$ lattice points $\{v_1, \cdots, v_{d+1}\}$ in $\bR^d$, where $v_i=(1+u_{i, q})e_i-qe$ for $1\leq i\leq d-1$, $v_d=-e_d-qe$ and $v_{d+1}=e_d-qe$. Then the barycentric coordinates of $\0$ with respect to $S$ are $\frac{q}{1+u_{1, q}}, \cdots, \frac{q}{1+u_{d-1, q}}, \frac{q}{2u_{d, q}}, \frac{q}{2u_{d, q}}.$

We claim that $S$ is a $d$-dimensional $\frac{1}{q}$-lc Fano simplex. 

First, the vertices $v_i$ are primitive lattice points because $1+u_{i, q}\equiv1(\mod q)$ for $i\in\{1, \cdots, d-1\}$.

%Second, the vertices $v_i$ generate the lattice $\ZZ^d$ bacause the vertices $v_1, \cdots, v_{d-1}$ generate the subspace spanned by $e_1, \cdots, e_{d-1}$.

Second, we want to show that $\frac{1}{q}=\max\{t|\inte(tS)\cap\ZZ^d=\{\0\}\}$. Take a real number $q'$ such that $0<q'\leq q$. Suppose $v\in\inte(\frac{1}{q'}S)\cap\ZZ^d$ has barycentric coordinates $\lambda_1, \cdots, \lambda_{d+1}$, i.e., $\Sigma_{i=1}^{d+1}\lambda_i=1$, $\lambda_i>0, \forall i\in\{1, \cdots, d+1\}$, and $v=\Sigma_{i=1}^{d+1}\lambda_i\frac{1}{q'}v_i\in\ZZ^d$. Then for $i\in\{1, \cdots, d-1\}$, $\lambda_i=\frac{q'z_i+q}{1+u_{i, q}}$ for some $z_i\in\ZZ$, and $\lambda_{d+1}-\lambda_d=q'z$ for some $z\in\ZZ$.

If $q'=q$, then $\lambda_i>0$ implies $\lambda_i\geq\frac{q}{1+u_{i, q}}$
for $i\in\{1, \cdots, d-1\}$. By \eqref{sum1} and the assumption that $v$ is an interior lattice point of $\frac{1}{q}S$, we must have $\lambda_i=\frac{q}{1+u_{i, q}}$ for $i\in\{1, \cdots, d-1\}$. And then $\lambda_d+\lambda_{d+1}=\frac{q}{u_{d, q}}$, which implies $\lambda_d=\lambda_{d+1}=\frac{q}{2u_{d, q}}$. So we conclude that $\inte(\frac{1}{q}S)\cap\ZZ^d=\{\0\}$.

If $0<q'<q$, then we can show that $-e_{d-1}\in\inte(\frac{1}{q'}S)\cap\ZZ^d$ by taking $\lambda_i=\frac{q}{1+u_{i, q}}$
for $i\in\{1, \cdots, d-2\}$, $\lambda_{d-1}=\frac{q-q'}{1+u_{d-1, q}}$ and $\lambda_d=\lambda_{d+1}=\frac{q'}{2(1+u_{d-1, q})}+\frac{q}{2u_{d, q}}$.

By Conrad's theorem (see \cite{Con02}), $S$ corresponds to the weighted projective space $$\bP(w_1, \cdots, w_{d+1}),\quad\text{where}\quad w_i:=|\det(v_j: j\in \{1, \cdots, d+1\}\setminus\{i\})|.$$
By calculations, $w_i=\frac{2u_{d, q}}{1+u_{i, q}}$ for $1\leq i\leq d-1$ and $w_d=w_{d+1}=1$. Let $X:=\bP(w_1, \cdots, w_{d+1})$, then $X$ is a $d$-dimensional $\frac{1}{q}$-lc toric Fano variety with $(-K_X)^d=\frac{2u_{d, q}^2}{q^{d+1}}.$  
\end{example}

\begin{prop}\label{ex2}
Let $e_1, \cdots, e_d$ be the standard basis of $\bR^d$. Set $e=\Sigma_{i=1}^{d-1}e_i$. Let $S:=\conv((1+u_{1, q})e_1-qe, \cdots, (1+u_{d-1, q})e_{d-1}-qe, -e_d-qe, e_d-qe)$. Then $S^*\simeq\conv(\frac{1+u_{1, q}}{q}e_{1}, \cdots, \frac{1+u_{d-1, q}}{q}e_{d-1}, \mp\frac{u_{d, q}}{q}e_{d})$.
\end{prop}
\begin{proof}
    Let $\beta_1\geq\cdots\geq\beta_{d+1}$ be the barycentric coordinates of $\0$ with respect to $S$. Then 
    $$(\beta_1, \cdots, \beta_{d+1})=(\frac{q}{1+u_{1, q}}, \cdots, \frac{q}{1+u_{d-1, q}}, \frac{q}{2u_{d, q}}, \frac{q}{2u_{d, q}}).$$
    Set $u_i:=\frac{1}{q}e_i$ for $1\leq i\leq d-1$ and
    $$u_d:=-\frac{u_{d, q}}{q}e_d-\frac{1}{q}\Sigma_{j=1}^{d-1}\frac{u_{d, q}}{1+u_{j, q}}e_j, u_{d+1}:=\frac{u_{d, q}}{q}e_d-\frac{1}{q}\Sigma_{j=1}^{d-1}\frac{u_{d, q}}{1+u_{j, q}}e_j.$$
    We claim that $S^*=\conv(u_1, \cdots, u_{d+1})$. According to the standard duality (see \cite[Proposition~6.1]{AKN16} and \cite[Proposition~3.6]{Nill07}), it suffices to show that $<(1+u_{i, q})e_i-qe, u_j>=-1$ for $i\in\{1, \cdots, d-1\}$ and $j\in\{1, \cdots, d+1\}\setminus\{i\}$, $<-e_d-qe, u_j>=-1$ for $j\in\{1, \cdots, d+1\}\setminus\{d\}$, and $<e_d-qe, u_j>=-1$ for $j\in\{1, \cdots, d+1\}\setminus\{d+1\}$. So the claim follows directly from the construction.
    %$<1+u_{i, q})e_i-qe, u_i>=\frac{1}{\beta_i}-1$ for $i\in\{1, \cdots, d-1\}$, $<-e_d-qe, u_d>=\frac{1}{\beta_d}-1$ and $<e_d-qe, u_{d+1}>=\frac{1}{\beta_{d+1}}-1$.
    
    Set $Q:=S^*+\frac{1}{q}\Sigma_{j=1}^{d-1}\frac{u_{d, q}}{1+u_{j, q}}e_j$. We need to show that 
    $$Q\simeq\conv(\frac{1+u_{1, q}}{q}e_{1}, \cdots, \frac{1+u_{d-1, q}}{q}e_{d-1}, \mp\frac{u_{d, q}}{q}e_{d}).$$

    Define $\phi: \bR^d\longrightarrow\bR^d$ by sending $e_i$ to $\frac{q}{1+u_{i, q}}(\frac{1}{q}e_i+\frac{1}{q}\Sigma_{j=1}^{d-1}\frac{u_{d, q}}{1+u_{j, q}}e_j)$ for $i\in\{1, \cdots, d-1\}$, and mapping $e_d$ to $e_d$. We claim that $\phi$ is an integral transformation.
    By \eqref{product}, we have 
    $$u_{d, q}=q\Pi_{l=1}^{d-1}(1+u_{l, q})=u_{i, q}(1+u_{i, q})\Pi_{l=i+1}^{d-1}(1+u_{l, q}).$$
    So we can rewrite $\phi(e_i)$ as $$\frac{1+u_{i, q}\Pi_{l=i+1}^{d-1}(1+u_{l, q})}{1+u_{i, q}}e_i+\Sigma_{1\leq j\leq d-1, j\neq i}\frac{u_{d, q}}{(1+u_{i, q})(1+u_{j, q})}e_j$$ 
    for $i\in\{1, \cdots, d-1\}$. Note that $1+u_{l, q}\equiv1(\mod 1+u_{i, q})$ for $i+1\leq l\leq d-1$. So $\frac{1+u_{i, q}\Pi_{l=i+1}^{d-1}(1+u_{l, q})}{1+u_{i, q}}\in\ZZ$. Thus the claim follows from the fact that  $\frac{u_{d, q}}{(1+u_{i, q})(1+u_{j, q})}\in \ZZ$ for $j\in\{1, \cdots, d-1\}$ with $j\neq i$. 

   Define $\psi: \bR^d\longrightarrow\bR^d$ by sending $e_i$ to $(1+u_{i, q})e_i-qe$ and mapping $e_d$ to $e_d$. Certainly, $\psi$ is an integral transformation and for $i\in\{1, \cdots, d-1\}$, we have
    \begin{align*}
        \psi\circ\phi(e_i)&{}=\frac{q}{1+u_{i, q}}(\frac{1+u_{i, q}}{q}e_i-e+\frac{u_{d, q}}{q}e-\Sigma_{j=1}^{d-1}\frac{u_{d, q}}{1+u_{j, q}}e)\\
        &{}=e_i,
    \end{align*}
    where the last equality follows from \eqref{sum1}.
    So $Q$ is unimodularly equivalent to $\conv(\frac{1+u_{1, q}}{q}e_{1}, \cdots, \frac{1+u_{d-1, q}}{q}e_{d-1}, \mp\frac{u_{d, q}}{q}e_{d})$, and hence
    $$S^*\simeq\conv(\frac{1+u_{1, q}}{q}e_{1}, \cdots, \frac{1+u_{d-1, q}}{q}e_{d-1}, \mp\frac{u_{d, q}}{q}e_{d}).$$
\end{proof}

Now we come to finish the proof of Theorem~\ref{bd for the dual vol of s}. 

\begin{lem}\label{unique1}
 Let $e_1, \cdots, e_d$ be the standard basis of $\bR^d$. Let $S\subset\bR^d$ be the lattice simplex $\conv(v_1, \cdots, v_{d+1})$ with $v_{d+1}=\0$. Suppose $\inte(S)\cap q\ZZ^d=\{\Sigma_{i=1}^{d}\frac{qv_i}{1+u_{i, q}}\}$. Then $S\simeq\conv((1+u_{1, q})e_1, \cdots, (1+u_{d, q})e_d, \0)$.
\end{lem}
\begin{proof}
    By assumption, we have $\Sigma_{i=1}^{d}\frac{v_i}{1+u_{i, q}}\in\ZZ^d$. Then for $1\leq i\leq d$, we conclude $\frac{v_i}{1+u_{i, q}}\in\ZZ^d$ by the observation that $1+u_{i, q}$ are coprime to each other. 

    Let $\Lambda:=\frac{v_1}{1+u_{1, q}}\ZZ+\cdots+\frac{v_d}{1+u_{d, q}}\ZZ$ be the sublattice of $\ZZ^d$. We claim that $\det(\Lambda)=1$. Since by construction,
    \begin{align*}
        \{\Sigma_{i=1}^{d}\frac{qv_i}{1+u_{i, q}}\}{}&\subset((0, 1]\frac{qv_1}{1+u_{1, q}}+\cdots+(0, 1]\frac{qv_d}{1+u_{d, q}})\cap q\Lambda\\
        {}&\subset((0, 1]\frac{qv_1}{1+u_{1, q}}+\cdots+(0, 1]\frac{qv_d}{1+u_{d, q}})\cap q\ZZ^d\\
        {}&\subset\inte(S)\cap q\ZZ^d=\{\Sigma_{i=1}^{d}\frac{qv_i}{1+u_{i, q}}\},
    \end{align*}
    we get $((0, 1]\frac{v_1}{1+u_{1, q}}+\cdots+(0, 1]\frac{v_d}{1+u_{d, q}})\cap \ZZ^d=\{\Sigma_{i=1}^{d}\frac{v_i}{1+u_{i, q}}\}$. So the claim holds because 
    $$\det\Lambda=\frac{\det\Lambda}{\det\ZZ^d}=|((0, 1]\frac{v_1}{1+u_{1, q}}+\cdots+(0, 1]\frac{v_d}{1+u_{d, q}})\cap \ZZ^d|$$ 
    by \cite[Chapter~VII, Theorem~2.5]{Bar02} and by noting that
    \begin{align*}
        &{}|((0, 1]\frac{v_1}{1+u_{1, q}}+\cdots+(0, 1]\frac{v_d}{1+u_{d, q}})\cap \ZZ^d|\\
        =&{}|([0, 1)\frac{v_1}{1+u_{1, q}}+\cdots+[0, 1)\frac{v_d}{1+u_{d, q}})\cap \ZZ^d|.
    \end{align*}
    Hence $\Lambda=\ZZ^d$. Then the map defined by sending $\frac{v_i}{1+u_{i, q}}$ to $e_i$ for $1\leq i\leq d$ is a unimodular transformation. Thus $S\simeq\conv((1+u_{1, q})e_1, \cdots, (1+u_{d, q})e_d, \0)$.
\end{proof}

\begin{lem}\label{unique2}
 Let $e_1, \cdots, e_d$ be the standard basis of $\bR^d$. Suppose $S\subset\bR^d$ is a $d$-dimensional $\frac{1}{q}$-lc Fano simplex such that the barycentric coordinates of $\0$ with respect to $S$ are $$(\beta_1, \cdots, \beta_{d+1})=(\frac{q}{1+u_{1, q}}, \cdots, \frac{q}{1+u_{d-1, q}}, \frac{q}{2u_{d, q}}, \frac{q}{2u_{d, q}}),$$ Then  $S\simeq\conv((1+u_{1, q})e_1, \cdots, (1+u_{d-1, q})e_{d-1}, \mp(v^{(d-1)}, h))$ for some $v^{(d-1)}\in\{0, 1, \cdots, h-1\}^{d-1}$ and some positive integer $h\in\ZZ$. 
\end{lem}

\begin{proof}
Let $v_1, \cdots, v_{d+1}$ be the vertices of $S$ associated to the  barycentric coordinates $\beta_1, \cdots, \beta_{d+1}$. Then $\Sigma_{i=1}^{d+1}\beta_i v_i=\0.$ So we get
    $$\frac{q}{1+u_{1, q}}\cdot v_1+\cdots+\frac{q}{1+u_{d-1, q}}\cdot v_{d-1}+\frac{q}{2u_{d, q}}\cdot(v_d+v_{d+1})=\0.$$
Since $u_{d, q}$ is divided by $q(1+u_{i, q})$ for each $1\leq i\leq d-1$, we have  $$\frac{v_d+v_{d+1}}{2}=-\Sigma_{i=1}^{d-1}\frac{u_{d, q}}{1+u_{i, q}}\cdot v_i \in q\ZZ^d.$$
Now set $v:=-\frac{v_d+v_{d+1}}{2}$, $S^\prime:=S+v$ and $v_i^\prime:=v_i+v$ for $1\leq i\leq d+1$. Denote by $H$ the hyperplane section $\conv(v_1^\prime, \cdots, v_{d-1}^\prime, \0)$ of $S^\prime$. By assumption that $S$ is a $\frac{1}{q}$-lc Fano simplex, we have $\inte(S)\cap q\ZZ^d=\{\0\},$
which implies $$\inte(S^\prime)\cap q\ZZ^d=\{v\}=\{\Sigma_{i=1}^{d-1}\frac{q}{1+u_{i, q}}v_i^\prime\}$$ by noting that $v\in q\ZZ^d$. 
And then $\{v\}\subset\inte(H)\subset\inte(S^\prime)$ implies $$\inte(H)\cap q\ZZ^d=\{\Sigma_{i=1}^{d-1}\frac{q}{1+u_{i, q}}v_i^\prime\}.$$ 
So Lemma~\ref{unique1} applies to $H$. We may assume $H=\conv((1+u_{1, q})e_1, \cdots, (1+u_{d-1, q})e_{d-1}, \0)$ by doing some unimodular transformation. 
By construction, $v_d^\prime+v_{d+1}^\prime=\0$. So we can write $-v_d^\prime=v_{d+1}^\prime=(v^{(d-1)}, h)$ for some $v^{(d-1)}\in\ZZ^{d-1}$ and $h\in\ZZ$. We may assume $h>0$ by possibly interchanging  $v_d^\prime$ and $v_{d+1}^\prime$. 
Let $v'^{(d-1)}\in\{0, 1, \cdots, h-1\}^{d-1}$ such that $v'^{(d-1)}\equiv v^{(d-1)} (\mod{h})$. Assume $v^{(d-1)}=(x_1, \cdots, x_{d-1})$ and $v'^{(d-1)}=(x_1+c_1 h, \cdots, x_{d-1}+c_{d-1}h)$ for some integers $c_1, \cdots, c_{d-1}$. Consider the $d\times d$ matrix $B:=$
    \[
\begin{bmatrix}
  1 & 0 & \cdots & 0 & 0 \\
  0 & 1 & \cdots & 0 & 0 \\
  \vdots & \vdots & \ddots & \vdots & \vdots \\
  0 & 0 & \cdots & 1 & 0 \\
  c_1 & c_2 & \cdots & c_{d-1} & 1
\end{bmatrix}
\]
We conclude the proof by the observation that the transformation induced by $B$ is a unimodular transformation preserving $\bR^{d-1}\times\{0\}$ which maps $S^\prime$ to $\conv((1+u_{1, q})e_1, \cdots, (1+u_{d-1, q})e_{d-1}, \mp(v'^{(d-1)}, h))$.
    
\end{proof}

\begin{proof}[{Proof of Theorem~\ref{bd for the dual vol of s}$(c)$}]
We only need to prove that for $q\geq2$, $\Vol(S^*)=\frac{2u_{d, q}^2}{q^{d+1}}$  implies $S^*\simeq\conv(\frac{1+u_{1, q}}{q}e_{1}, \cdots, \frac{1+u_{d-1, q}}{q}e_{d-1}, \mp\frac{u_{d, q}}{q}e_{d}).$ Suppose the equality holds. By Theorem~\ref{lower bd for bc d}, we see that the equality holds if only if 
$$(\beta_1, \cdots, \beta_{d+1})=(\frac{q}{1+u_{1, q}}, \cdots, \frac{q}{1+u_{d-1, q}}, \frac{q}{2u_{d, q}}, \frac{q}{2u_{d, q}}),$$
where $\beta_1, \cdots, \beta_{d+1}$ are the barycentric coordinates of $\0$ with respect to $S$. So by applying Lemma~\ref{unique2}, we have  $S\simeq\conv((1+u_{1, q})e_1, \cdots, (1+u_{d-1, q})e_{d-1}, \mp(v^{(d-1)}, h))$ for some $v^{(d-1)}\in\{0, 1, \cdots, h-1\}^{d-1}$ and some positive integer $h\in\ZZ$. On the other hand, deducing from \cite[Proposition~3.6]{Nill07} or following from \cite[Proposition~6.2]{AKN16}, we have $$\Vol(S)\cdot\Vol(S^*)=\frac{1}{\Pi_{i=1}^{d+1}\beta_i}.$$
So by calculation, $\Vol(S)=\frac{2u_{d, q}}{q}=\Pi_{i=1}^{d-1}(1+u_{i, q})\cdot2h$. Then we must have $h=1$ and hence $S\simeq\conv((1+u_{1, q})e_1, \cdots, (1+u_{d-1, q})e_{d-1}, \mp e_d)$. So the conclusion follows from Proposition~\ref{ex2}.
\end{proof}

\section{Bounding the dual normalized volumes of Fano polytopes}
In this section, we follow the strategy in \cite{BKN22} to derive an upper bound for the dual normalized volumes of minimal Fano polytopes. Fix two positive integers $d\geq3$ and $q\geq2$. Let $P$ be a $d$-dimensional minimal $\frac{1}{q}$-lc Fano polytope such that $P$ is not a simplex. Then there exists a decomposition of $P$ into lower dimensional $\frac{1}{q}$-lc Fano simplices as in Theorem~\ref{decomp}. Keep the same notation as in Theorem~\ref{decomp}. According to the monotonicity of the normalised volume (cf. \cite[(3.1)]{BKN22}), we have the following upper bound for the normalized volume of $P^*$:
\begin{align}
    \Vol(P^*)\leq\frac{(d_1+\cdots+d_t)!}{d_{1}!\cdots d_{t}!}\cdot\Pi_{i=1}^{t}\Vol(S_i^*),\label{ineq for dual vol}
\end{align}
where for each $1\leq i\leq t$, $\Vol(S_i^*)$ means the normalized volume of $S_i^*$ in the corresponding $d_i$-dimensional $\bR$-vector space. By Theorem~\ref{bd for the dual vol of s}, we have $\Vol(S_i^*)\leq \frac{2u_{d_i, q}^{2}}{q^{d_i+1}}$.
    %$$B_{d_i}:=
%\begin{cases}
% & \text{if } d_i=1;\\
%\max\{9, \frac{2u_{2, q}^2}{q^3}\}  & \text{if } d_i=2;\\
%\frac{2u_{i, q}^{2}}{q^{i+1}} & \text{if }  d_i\neq1, 2.
%\end{cases}
%$$
Our first step is to figure out when the following inequality holds:
\begin{align}
    \frac{(d_1+\cdots+d_t)!}{d_{1}!\cdots d_{t}!}\cdot\Pi_{i=1}^{t}\frac{2u_{d_i, q}^{2}}{q^{d_i+1}}<\frac{2u_{d, q}^2}{q^{d+1}}.\label{ineq for dual vol+}
\end{align}
For this purpose, we shall find a reasonable approximation of $u_{n, q}$ for any $n\in\ZZ_{>0}$.
\begin{lem}\label{appro}
    Fix a positive integer $q\geq2$. Then there exists a real number $K$ satisfying $\sqrt{2}<K<q$ such that for any $n\in\ZZ_{>0}$, 
    $$u_{n, q}<K^{2^n}<u_{n, q}+1.$$
\end{lem}

\begin{proof}
    Set $u_{0, q}:=1$ and $y_n:=\log u_{n, q}$ for $n\geq0$. Set $c_n:=\log(1+\frac{1}{u_{n, q}})$ for $n\geq1$ and $c_0:=\log q$. Then $y_{n+1}=2y_n+c_n$ for any $n\geq0$. So for any $n\in\ZZ_{>0}$, we have
    \begin{align*}
        y_n &{}=2^n\cdot(y_0+\frac{c_0}{2}+\frac{c_1}{2^2}+\cdots+\frac{c_{n-1}}{2^n})\\
        &{}=\Sigma_{i=0}^{\infty}2^{n-1-i}c_i-\Sigma_{i=n}^{\infty}2^{n-1-i}c_i\\
        &{}=x_n-r_n,
    \end{align*}
    where $x_n:=\Sigma_{i=0}^{\infty}2^{n-1-i}c_i$ and $r_n:=\Sigma_{i=n}^{\infty}2^{n-1-i}c_i$.
    Choose $K:=e^{\Sigma_{i=0}^{\infty}2^{-1-i}c_i}$. Then $K^{2^n}=e^{x_n}$. Since $\frac{c_0}{2}<\Sigma_{i=0}^{\infty}2^{-1-i}c_i<\Sigma_{i=0}^{\infty}2^{-1-i}c_0=c_0$, we have $$\sqrt{2}\leq\sqrt{q}<K<q.$$

    For $n\geq1$, $0<r_n<\Sigma_{i=n}^{\infty}2^{n-1-i}c_n\leq c_n$ by construction, so we have $$K^{2^n}e^{-c_n}<u_{n, q}=e^{x_n-r_n}<K^{2^n},$$ which implies $u_{n, q}<K^{2^n}<u_{n, q}+1.$
\end{proof}

\begin{prop}\label{1st bd}
    Fix two positive integers $d\geq3$ and $q\geq2$. Let $P$ be a $d$-dimensional minimal $\frac{1}{q}$-lc Fano polytope. Suppose $P$ is not a simplex and $P$ is decomposed into lower dimensional $\frac{1}{q}$-lc Fano simplices as in Theorem~\ref{decomp}. Keep the same notation as in Theorem~\ref{decomp}. Suppose the decomposition is not the case when $t=2, d_1=d_2=d-1$. Then $$\Vol(P^*)<\frac{2u_{d, q}^2}{q^{d+1}}.$$
\end{prop}
\begin{proof}

Due to \eqref{ineq for dual vol} and Theorem~\ref{bd for the dual vol of s}, we only need to prove \eqref{ineq for dual vol+} holds. We separate the proof into two parts: $t\geq3$ and $t=2$. 

\textbf{Case $(1)$: $t\geq3$}. 

By \eqref{bd di}, $d_i\leq d-t+1$. $\frac{(d_1+\cdots+d_t)!}{d_{1}!\cdots d_{t}!}\cdot\Pi_{i=1}^{t}\frac{2u_{d_i, q}^{2}}{q^{d_i+1}}$ will increase if $d_i$ increase. So to prove \eqref{ineq for dual vol+}, it is enough to show 
    \begin{align}
    \frac{(t(d-t+1))!}{(d-t+1)!^{t}}\cdot (\frac{2u_{d-t+1, q}^{2}}{q^{d-t+2}})^t<\frac{2u_{d, q}^2}{q^{d+1}}.\label{ineq for dual vol++}
    \end{align}
 Since $q\geq2$, \eqref{ineq for dual vol++} holds if
    \begin{align*}
    \frac{(t(d-t+1))!}{(d-t+1)!^{t}}\cdot \frac{u_{d-t+1, q}^{2t}}{q^{t(d-t+1)-d}}< u_{d, q}^2
    \end{align*}
    holds.
    Set $\lambda:=t(d-t+1)$. Since $d\geq3$ and $3\leq t\leq d$, we have $$3\leq d\leq\lambda\leq\frac{(d+1)^2}{4}.$$
    Take $K$ as in Lemma~\ref{appro}. Then $\sqrt{2}<K<q$ and $u_{n, q}<K^{2^n}<u_{n, q}+1$. 
    So \eqref{ineq for dual vol++} holds if 
    \begin{align}
    \frac{\lambda!}{(d-t+1)!^{t}}\cdot K^{2^{d-t+2}\cdot t+d-\lambda}\leq K^{2^{d+1}-2}\label{ineq for dual vol3+}
    \end{align}
    holds. Here $u_{d, q}$ is reduced to $K^{2^{d}-1}$.
By using the relation $\lambda!<2\cdot2^2\cdots2^{\lambda-1}=2^{\frac{1}{2}\lambda(\lambda-1)}< K^{\lambda(\lambda-1)}$ and taking $\log_K$, 
    \eqref{ineq for dual vol3+} is reduced to 
    \begin{align}
    \lambda(\lambda-2)\leq 2^{d+1}\cdot(1-2^{1-t}\cdot t)-2-d.\label{ineq for dual vol4+}
    \end{align}
By noting that $\lambda(\lambda-2)\leq\frac{(d+1)^2}{4}(\frac{(d+1)^2}{4}-2)$ and $1-2^{1-t}\cdot t$ increases as $t\geq3$ increases, one can check directly that inequality \eqref{ineq for dual vol4+} holds when 
$d$ and $t$ satisfy the following conditions:
\begin{enumerate}
    \item[1.]  ~$d\geq12 \text{~and~} t\geq3$;
    \item[2.]  ~$d\in\{11, 10\} \text{~and~} t\geq4$;
    \item[3.]  ~$d=9 \text{~and~} t\geq5$;
    \item[4.]  ~$d=8 \text{~and~} t\geq6$.
\end{enumerate}
Now we need to make a better estimate of this quantity $\frac{\lambda!}{(d-t+1)!^{t}}$ to show \eqref{ineq for dual vol3+} holds under some conditions. Set $Q(d, t):=\frac{\lambda!}{(d-t+1)!^{t}}$, $L(d, t):=Q(d, t)\cdot K^{2^{d-t+2}\cdot t+d-\lambda}$ and $R(d):=K^{2^{d+1}-2}$. By estimating $Q(d, t)$ appropriately, we have the following list:

\begin{align*}
    &{}d=11, &{}t=3,\quad  &{}Q(d, t)<2^{38}<K^{76}, &{}L(d, t)<K^{76+2^{10}\cdot3}<R(11);\\
    &{}d=10, &{}t=3, \quad &{}Q(d, t)<2^{34}<K^{68}, &{}L(d, t)<K^{68+2^{9}\cdot3}<R(10);\\
    &{}d=9, &{}t=3, 4, \quad &{}Q(d, t)<2^{42}<K^{84}, &{}L(d, t)<K^{84+2^{8}\cdot3}<R(9);\\
    &{}d=8, &{}t=3, 4, 5, \quad&{}Q(d, t)<2^{39}<K^{78}, &{}L(d, t)<K^{78+2^{7}\cdot3}<R(8);\\
    &{}d=7, &{}t\geq3, \quad &{}Q(d, t)<2^{28}<K^{56}, &{}L(d, t)<K^{56+2^{6}\cdot3}<R(7);\\
    &{}d=6, &{}t\geq3, \quad &{}L(d, t)<K^{122}<R(6);\\
    &{}d=5, &{}t\geq4, \quad &{}L(d, t)<K^{56}<R(5);\\
    &{}d=4, &{}t=4, \quad &{}L(d, t)<K^{26}<R(4).\\
\end{align*}
In the above list, it should be noted that we use the relation $K^{2^{d-t+2}\cdot t+d-\lambda}\leq K^{2^{d-1}\cdot3}$ in the estimation process when $7\leq d\leq11$. For $4\leq d\leq6$, it is just a direct comparison by estimating $Q(d, t)$ appropriately.

%If $d=11$, then $\lambda(\lambda-2)\leq1260$, and hence inequality \eqref{ineq for dual vol4+} holds if $t\geq4$. For $t=3$, $\frac{\lambda!}{(d-t+1)!^{t}}=\frac{27!}{9!^3}<2^{38}<K^{76}$ implies inequality \eqref{ineq for dual vol3+} holds.

%If $d=10$, then $\lambda(\lambda-2)\leq840$, and hence inequality \eqref{ineq for dual vol4+} holds if $t\geq4$. For $t=3$, $\frac{\lambda!}{(d-t+1)!^{t}}=\frac{24!}{8!^3}<2^{34}<K^{68}$ implies inequality \eqref{ineq for dual vol3+} holds.

%If $d=9$, then $\lambda(\lambda-2)\leq575$, and hence inequality \eqref{ineq for dual vol4+} holds if $t\geq5$. For $3\leq t\leq4$, $\frac{\lambda!}{(d-t+1)!^{t}}\leq\frac{24!}{6!^4}<2^{42}<K^{84}$ implies inequality \eqref{ineq for dual vol3+} holds.

%If $d=8$, then $\lambda(\lambda-2)\leq360$, and hence inequality \eqref{ineq for dual vol4+} holds if $t\geq6$. For $3\leq t\leq5$, $\frac{\lambda!}{(d-t+1)!^{t}}\leq\frac{20!}{4!^5}<2^{39}<K^{78}$ implies inequality \eqref{ineq for dual vol3+} holds.

%If $d=7$, then $\frac{\lambda!}{(d-t+1)!^{t}}\leq\frac{15!}{3!^5}<2^{28}<K^{56}$ implies inequality \eqref{ineq for dual vol3+} holds.

%If $d=6$, then the left hand side of \eqref{ineq for dual vol3+} is less than $K^{122}$, which is less than the right hand side of \eqref{ineq for dual vol3+}.

Following from the above discussion, it remains to prove \eqref{ineq for dual vol++} holds when $3\leq d\leq5$ and $t=3$.
Rewrite \eqref{ineq for dual vol++} as
    \begin{align*}
    \frac{(t(d-t+1))!}{(d-t+1)!^{t}}\cdot \frac{2^t\cdot u_{d-t+1, q}^{2t}}{q^{t(d-t+2)}}<\frac{2u_{d, q}^2}{q^{d+1}}.
    \end{align*}
Denote by $L(d, t, q)$ the left hand side of the above inequality and $R(d, t, q)$ the right hand side of the above inequality.

If $d=5, t=3$, then $L(d, t, q)=\frac{2^2\cdot9!}{3!^3}\cdot\frac{2u_{3, q}^{6}}{q^{12}}<2^{13}\cdot\frac{2u_{3, q}^{6}}{q^{12}}\leq2u_{3, q}^6\cdot q$, and $R(d, t, q)=\frac{2u_{5, q}^2}{q^6}>\frac{2u_{3, q}^8}{q^6}$. Since $u_{3, q}^{2}>q^8$, $L(d, t, q)<R(d, t, q)$.

If $d=4, t=3$, then $L(d, t, q)=360\cdot\frac{2u_{2, q}^{6}}{q^9}$, and $R(d, t, q)=\frac{2u_{4, q}^2}{q^5}$. For $q=2$, we have $u_{2, 2}=6$, $u_{4, 2}=1806$, then $L(d, t, q)<R(d, t, q)$ by calculation; for $q\geq3$, we have the estimate $360<q^6$, hence $L(d, t, q)<\frac{2u_{2, q}^{6}}{q^3}<\frac{2u_{2, q}^{8}}{q^5}<R(d, t, q)$.

If $d=3, t=3$, then $L(d, t, q)=48$, and $R(d, t, q)=\frac{2u_{3, q}^2}{q^4}>2(1+u_{2, q})^2=2(1+q+q^2)^2\geq98$. Hence $L(d, t, q)<R(d, t, q)$.

\textbf{Case $(2)$: $t=2$}. 

In this case, for the same reason as in the first case, we only need to prove \eqref{ineq for dual vol+} under the condition that $d_1=d-1, d_2=d-2$, that is, it suffices to show
\begin{align*}
    \frac{(2d-3)!}{(d-1)!(d-2)!}\cdot \frac{2u_{d-1, q}^{2}}{q^{d}}\cdot \frac{2u_{d-2, q}^{2}}{q^{d-1}}<\frac{2u_{d, q}^{2}}{q^{d+1}}.
\end{align*}
Since $u_{d, q}=u_{d-1, q}(1+u_{d-1, q})$, we can rewrite the above inequality as
\begin{align}
    \frac{(2d-3)!}{(d-1)!(d-2)!}\cdot \frac{2u_{d-2, q}^2}{q^{d-2}}< (1+u_{d-1, q})^2.\label{ineq for dual vol5+}
\end{align}
Take $K$ as in Lemma~\ref{appro}. Since $\sqrt{2}<K<q$ and $u_{n, q}<K^{2^n}<u_{n, q}+1$, we have 
\begin{align*}
    \frac{(2d-3)!}{(d-1)!(d-2)!}\cdot\frac{2u_{d-2, q}^2}{q^{d-2}}&{}<2^{(2d-3)(d-2)}\cdot\frac{K^{2^{d-1}}}{q^{d-3}}\\
    &{}< K^{(2d-3)(2d-4)-d+3+2^{d-1}},
\end{align*}
and $(1+u_{d-1, q})^2>K^{2^d}$. By taking $\log_K$, proving \eqref{ineq for dual vol5+} is reduced to proving the following inequality
\begin{align*}
    (2d-3)(2d-4)-d+3\leq2^{d-1}.
\end{align*}
By calculations, the above inequality holds when $d\geq9$.

To finish the proof, we need to make a better estimate of this quantity
$\frac{(2d-3)!}{(d-1)!(d-2)!}$ when $3\leq d\leq8$. For $3\leq d\leq8$, it can be verified by estimating $\frac{(2d-3)!}{(d-1)!(d-2)!}$ appropriately that 
\begin{align*}
    \frac{(2d-3)!}{(d-1)!(d-2)!}\cdot \frac{2u_{d-2, q}^2}{q^{d-2}}<K^{2^d}<(1+u_{d-1, q})^2.
\end{align*}

%for $d=3$, $\frac{(2d-3)!}{(d-1)!(d-2)!}\cdot \frac{2u_{d-2, q}^2}{q^{d-2}}=6q<q^4<(1+u_{2, q})^2$.

\end{proof}

To bound the volume of $P^*$ when $P$ is decomposed into two $(d-1)$-dimensional $\frac{1}{q}$-lc Fano simplices, we apply the result figured out through integration in \cite{BKN22}.

\begin{lem}[{cf. \cite[$(6.1), (6.2)$]{BKN22}}]\label{integration}
Fix two integers $d\geq3$ and $q\geq2$. Let $P$ be a $d$-dimensional minimal $\frac{1}{q}$-lc Fano polytope. Suppose $P$ is not a simplex and is decomposed into two $(d-1)$-dimensional $\frac{1}{q}$-lc Fano simplices. Keep the same notation as in Theorem~\ref{decomp}. Suppose the barycentric coordinates of $\0$ with respect to $S_i$ are $\beta_{i, 1}, \cdots, \beta_{i, d}$ for each $i=1, 2$. Then
$$\Vol(P^*)\leq\frac{1}{\beta_{1, 1}\cdots\beta_{1, d}}+\frac{1}{\beta_{2, 1}\cdots\beta_{2, d}}.$$
Furthermore, if the equality holds, then for each $i=1, 2$, there exists two of $\beta_{i, 1}, \cdots, \beta_{i, d}$ are equal.
\end{lem}

\begin{prop}\label{2nd bd}
Fix two integers $d\geq3$ and $q\geq2$. Let $P$ be a $d$-dimensional minimal $\frac{1}{q}$-lc Fano polytope. Suppose $P$ is not a simplex and is decomposed into two $(d-1)$-dimensional $\frac{1}{q}$-lc Fano simplices. Then 
$$\Vol(P^*)<\frac{2u_{d, q}^2}{q^{d+1}}.$$
\end{prop}

\begin{proof}
Keep the same notation as in Theorem~\ref{decomp}. Suppose the barycentric coordinates of $\0$ with respect to $S_i$ are $\beta_{i, 1}, \cdots, \beta_{i, d}$ for each $i=1, 2$. Then by Lemma~\ref{integration}, we have
\begin{align}
    \Vol(P^*)\leq\frac{1}{\beta_{1, 1}\cdots\beta_{1, d}}+\frac{1}{\beta_{2, 1}\cdots\beta_{2, d}}.\label{eq1}
\end{align}

Applying Theorem~\ref{lower bd for bc d+1}, for each $i=1, 2$, we have

\begin{align}
    \frac{1}{\beta_{i, 1}\cdots\beta_{i, d}}\leq\frac{u_{d, q}^2}{q^{d+1}}.\label{eq2}
\end{align}

So we get the conclusion by noting that the equality in \eqref{eq1} and the equality in \eqref{eq2} cannot hold simultaneously.
\end{proof}

\section{Proof of main results}

\begin{proof}[Proof of Theorem~\ref{main thm}]
If $P_X$ is not a simplex, then $(-K_X)^d=\Vol(P_X^*)<\frac{2u_{d, q}^2}{q^{d+1}}$ by Theorem~\ref{bounds for d3}, Proposition~\ref{1st bd} and Proposition~\ref{2nd bd}. If $P_X$ is a simplex, then the conclusion follows from Theorem~\ref{bd volume p1}.
\end{proof}

\begin{proof}[Proof of Theorem~\ref{dual vol bd for lattice polytope}]
 We only need to consider the case when $P$ is a $d$-dimensional minimal $\frac{1}{q}$-lc Fano polytope. Then the conclusions follow from Theorem~\ref{bd for the dual vol of s}, Proposition~\ref{1st bd} and Proposition~\ref{2nd bd}.
\end{proof}

\begin{proof}[Proof of Theorem~\ref{vol bd for simplex}]
By assumption, $\frac{1}{q}S\cap\ZZ^d=\{\0\}$. By abuse of notation, we also assume $\beta_1\geq\cdots\geq\beta_{d+1}$ are the barycentric coordinates of $\0$ with respect to $S$. Then by applying Theorem~\ref{simplex vol bd}, we have 
$$\vol_{\ZZ}(\frac{1}{q}S)\leq\frac{1}{d!\beta_1\cdots\beta_d}.$$
An observation is that $(\beta_1, \cdots, \beta_{d+1})$ satisfies conditions \eqref{relation1}, \eqref{relation2} and \eqref{relation3} (see Remark~\ref{rem2}). So Theorem~\ref{lower bd for bc d} applies. We have 
\begin{align*}
    \vol_{\ZZ}(S)&{}\leq\frac{q^d}{d!\beta_1\cdots\beta_d}\\
    &{}\leq\frac{2u_{d, q}^2}{d!q}.
\end{align*}
To conclude the proof, we only need to show that $\inte(\frac{1}{q}S)\cap\ZZ^d=\{\0\}$, where $S=\conv((1+u_{1, q})e_1-qe, \cdots, (1+u_{d-1, q})e_{d-1}-qe, -u_{d, q}e_d-qe, u_{d, q}e_d-qe)$. Moreover, we claim that $\frac{1}{q}=\max\{t|\inte(tS)\cap\ZZ^d=\{\0\}\}$.

Take a real number $q'$ such that $0<q'\leq q$. Suppose $v\in\inte(\frac{1}{q'}S)\cap\ZZ^d$ has barycentric coordinates $\lambda_1, \cdots, \lambda_{d+1}$,  i.e., $$\Sigma_{i=1}^{d+1}\lambda_i=1, \lambda_i>0, \forall i\in\{1, \cdots, d+1\},$$ and $$v=\Sigma_{i=1}^{d-1}\frac{\lambda_i}{q'}((1+u_{i, q})e_i-qe)+\frac{\lambda_d}{q'}(-u_{d, q}e_d-qe)+\frac{\lambda_{d+1}}{q'}(u_{d, q}e_d-qe)\in\ZZ^d.$$ Then for $i\in\{1, \cdots, d-1\}$, $\lambda_i=\frac{q'z_i+q}{1+u_{i, q}}$ for some $z_i\in\ZZ$, and $\lambda_{d+1}-\lambda_d=\frac{q'z}{u_{d, q}}$ for some $z\in\ZZ$.

If $q'=q$, then $\lambda_i>0$ implies $\lambda_i\geq\frac{q}{1+u_{i, q}}$
for $i\in\{1, \cdots, d-1\}$. By \eqref{sum1} and the assumption that $v$ is an interior lattice point of $\frac{1}{q}S$, we must have $\lambda_i=\frac{q}{1+u_{i, q}}$ for $i\in\{1, \cdots, d-1\}$. And then $\lambda_d+\lambda_{d+1}=\frac{q}{u_{d, q}}$, which implies $\lambda_d=\lambda_{d+1}=\frac{q}{2u_{d, q}}$. So we conclude that $\inte(\frac{1}{q}S)\cap\ZZ^d=\{\0\}$.

If $0<q'<q$, then we can show that $e_d\in\inte(\frac{1}{q'}S)\cap\ZZ^d$ by taking $\lambda_i=\frac{q}{1+u_{i, q}}$
for $i\in\{1, \cdots, d-1\}$, $\lambda_d=\frac{q-q'}{2u_{d, q}}$ and $\lambda_{d+1}=\frac{q+q'}{2u_{d, q}}$.

\end{proof}

\section*{Acknowledgments} I would like to express my gratitude to Caucher Birkar for his suggestions and constant support. I am grateful to Florin Ambro for a lot of helpful discussions.

\end{document}